\documentclass[12pt]{article}

\usepackage{amsmath, amssymb, amsthm}  
\usepackage{geometry}  
\usepackage{enumitem}  
\usepackage{mathrsfs}  
\usepackage{graphicx}  
\usepackage{float} 
\usepackage{cite}      
\usepackage{color}     
\usepackage{tikz}      
\usepackage{authblk}   
\usepackage{tikz-cd}

\usepackage{authblk}

\newtheorem{theorem}{Theorem}[section]
\newtheorem{lemma}[theorem]{Lemma}
\newtheorem{proposition}[theorem]{Proposition}

\theoremstyle{definition}

\newtheorem{example}[theorem]{Example}

\theoremstyle{remark}
\newtheorem{remark}[theorem]{Remark}

\newcommand{\dC}{\mathbb{C}}
\newcommand{\dE}{\mathbb{E}}

\newcommand{\dN}{\mathbb{N}}
\newcommand{\dP}{\mathbb{P}}
\newcommand{\dR}{\mathbb{R}} 

\newcommand{\dT}{\mathbb{T}}
\newcommand{\dZ}{\mathbb{Z}}

\newcommand{\lB}{\mathcal{B}}

\newcommand{\lF}{\mathcal{F}}
\newcommand{\lH}{\mathcal{H}}

\newcommand{\lT}{\mathcal{T}}
\newcommand{\lX}{\mathcal{X}}

\newcommand{\dd}{\mathrm{d}}

\usepackage{color}

\title{Random Schr\"{o}dinger operators and convolution on wreath products}
\author{Adam Arras\thanks{The author was partially supported by the KKP 139502 project and the ERC Dynasnet 810115 project.}}
\date{}
\affil{HUN-REN Alfr\'ed R\'enyi Institute of Mathematics, Budapest, Hungary}

\begin{document}

\maketitle

\begin{abstract}
We establish a spectral correspondence between random Schr\"{o}dinger operators and deterministic convolution operators on wreath products, generalizing previous results that relate Lamplighter groups to Schr\"{o}dinger operators with Bernoulli potentials. Using this correspondence in both directions, we obtain an elementary criterion for the absolute continuity of convolutions on wreath product, Lifschitz tail estimates for Schr\"{o}dinger operators on Cayley graph of polynomial growth, and an exact formula for the second moment of the Green function, expressed in terms of the wreath product with an Abelian group of lamps.
\end{abstract}


\section{Introduction and main results}

The central theme of this paper is the spectral connection between deterministic convolution operators on wreath product groups and random Schr\"{o}dinger operators. We will mostly be concerned with Hilbert spaces of the form $\lH=\ell^2(\Gamma)$, $\Gamma$ being a countable, discrete group with identity $e_\Gamma$. Given $x\in \Gamma$, the delta function $\delta_x$ is defined by $ \delta_x(y) = 1 $ if $ y = x $, and $\delta_x(y) = 0$ otherwise. The spectral representation\cite[Theorem VIII.6]{reed1972methods} of a self-adjoint operator will be denoted $$L=\int_{\Sigma(L)} \lambda E_L(\dd \lambda), \qquad \mu_L^{\varphi,\psi}(\cdot) = \langle \varphi, E_L(\cdot) \psi\rangle.$$ We shorten the notation of the spectral measures by setting $\mu_L^x=\mu_L^{\delta_x,\delta_x}$.  

\paragraph{Convolution operators.} The \textit{left regular representation} $(L_x f)(y) = f(x^{-1}y)$, for any square summable function $f\in \ell^2(\Gamma)$, defines an unitary action of the group $\Gamma$. To any complex finite measure $m \in \mathcal{M}_{\dC}(\Gamma) \simeq \ell^1(\Gamma)$, we consider the sum $L_m := \sum_{x} m(x)L_x$. Any operators of this form will be referred as a \textit{convolution operators}. When $m$ is a probability measure, this correspond to the Markov operator associated with the simple random walk, where increments are i.i.d. and follow the distribution $m$. We refers to the book \cite{folland2016course} for a modern introduction to group convolution. A convolution operator is self-adjoint if and only if the associated measure satisfies $m(g^{-1}) = \bar{m}(g)$. In this case, the spectral measure $\mu_{L_m}^{x}$ exists and is independent of $x \in \Gamma$. This measure is commonly known as the Plancherel measure \cite{bordenave2016spectrum}, or the Kesten-von Neumann-Serre spectral measure \cite{grigorchuk2004ihara}. A case of particular of interest is when the measure takes value zero and one, corresponding to the adjacency operator of a Cayley graph. The spectral analysis of this algebraic structure, initiated by Kesten \cite{kesten1959symmetric}, has led to some surprising recent developments \cite{grigorchuk2001lamplighter}; see also \cite{bordenave2016spectrum,grigorchuk2021laplace,de2024spectra,mohar1989survey,linnell1998analytic,elek2003analytic}.

\paragraph{Random Schr\"{o}dinger operator.} In his seminal work \cite{anderson1958absence}, the physicist P. W. Anderson studied the transport properties of disordered crystals. This quantum analogue of random walks in a random environment is modeled as follows. The Cayley graph $\operatorname{Cay}(\Gamma,S)$ of an infinite, finitely generated group describes the hopping terms of a particle on the crystal lattice. The presence of small-scale impurities is represented by an independent and identically distributed scalar field $(v(x, \omega))_{x\in \Gamma}$. To avoid technicalities, we only consider distribution with bounded support. The system evolves according to the unitary group generated by the \textit{Random Schr\"{o}dinger operator} $H(\omega)=A+V$ of form \textit{Laplacian plus potential\footnote{Instead of $A=L_{1_S}$ where $1_S = \sum_{S} \delta_s$, the true Laplacian of a Cayley graph $\operatorname{Cay}(\Gamma,S)$ is the convolution operator $\Delta = |S|I-A$.}}, which acts on $\varphi \in \ell^2(\Gamma)$ via 
\begin{equation}\label{eq:defSchrOp}
  (H(\omega)\varphi)(x) = \sum_{y \sim x} \varphi(y) + v(x,\omega) \varphi(x).  
\end{equation}
Since the distribution of the operator is invariant by the action of $\Gamma$ on $\ell^2(\Gamma)$, it follows from ergodicity theory that global features of $H(\omega)$ are almost sure. For example, the spectrum verifies $\Sigma(H(\omega)) \overset{a.s.}{=}\Sigma_H$, for some closed set $\Sigma_H$, called the almost sure spectrum. The same applies to any subset in the Lebesgue decomposition of the spectral measures $\mu_{H(\omega)}^{x}$ into absolutely continuous, singular continuous and pure point part. We refer to \cite[Theorem 3.10]{aizenman2015random} for more details. In the one dimensional case $\Gamma = \dZ$, it turns out that arbitrarily small randomness completely breaks down the transport properties of the medium. This translates into pure point spectrum and exponential decay of the eigenfunctions for the operator $H(\omega)$. This phenomenon, that also occurs for any group $\Gamma$ for high enough disorder, is referred as Anderson’s localization \cite{stolz2011introduction}. It is by now reasonably understood from a mathematical viewpoint, since the work of Fröhlich and Spencer (via multi-scale analysis \cite{frohlich1983absence}) and Aizenman and Molchanov (via fractionnal moment method \cite{aizenman1993localization}, see also \cite{graf1994anderson}). In contrary, the conjectured persistence at low disorder of transport in the case of $\Gamma=\dZ^d,d\geq3$, and therefore the stability of an absolutely continuous spectrum, remains a challenging open question despite considerable efforts of the community. Some results are known in the case where the potential decays at infinity \cite{bourgain2001random,bourgain2003random}, but the presence of absolutely continuous spectrum for a stationary potential has only been proven for trees. This has been achieved on the $d$-regular tree $\dT_d$ by Klein \cite{klein1998extended}, and generalized by different approaches since then \cite{froese2007absolutely,aizenman2006stability,keller2012absolutely,arras2023existence}.
\paragraph{Wreath product of group.} The wreath product $\Lambda \wr \Gamma$ of two discrete groups is defined as the semidirect product $(\bigoplus_{\Gamma} \Lambda) \rtimes \Gamma$. The factor $\bigoplus_{\Gamma} \Lambda$ denotes the direct sum, the subgroup of the direct product $\prod_\Gamma \Lambda$ that consists of elements with finite support. The factor $\Gamma$ acts on $\bigoplus_{\Gamma} \Lambda$ by shifting indices via the regular action. Elements of $\Lambda \wr \Gamma$ can thus be represented as pairs $(f, g)$, where $f: \Gamma \to \Lambda$ has finite support\footnote{In our context, the support corresponds to the set $f^{-1}(\Lambda \setminus \{e_\Lambda\})$.} and  $g \in \Gamma$. The group operations is given by  
$$
(f, g) \cdot (f', g') = (f \cdot (f' \circ g^{-1}), g g').
$$  
Here, $f' \circ g^{-1}$ denotes the shifted function $h \mapsto f'(g^{-1}h)$, reflecting the semidirect action. This algebraic structure has an intuitive geometric interpretation as automorphisms on a two-level rooted tree $\lT_{\Gamma,\Lambda}$ constructed as follows. The root $o$ is connected to each element $g$ of $\Gamma$ and each $g$ is attached to a copy of $\Lambda$, denoted $gl,l\in \Lambda$.
\begin{figure}[H]
\centering
\begin{tikzcd}
                         &        & o \arrow[lldd] \arrow[dd] \arrow[rrdd] &        &                          &        \\
                         &        &                                                &        &                          &        \\
g_1=e_\Gamma \arrow[d] \arrow[rd] &        & g_2 \arrow[d] \arrow[rd]                       & \ldots & g_n \arrow[d] \arrow[rd] &        \\
g_1l_1\ldots             & g_1l_p & g_2l_1\ldots                                   & g_2l_p & 
g_nl_1\ldots             & g_nl_p
\end{tikzcd}
\\ \,

    \caption{The tree $\lT_{\Gamma,\Lambda}$}
    \label{fig:TGH}
\end{figure} 
An automorphism is a permutation of the vertices that preserves the adjacency relation. In particular, it must fix the root. Given an element $h=(f,g) \in \Lambda \wr \Gamma$, one associates the automorphism $\varphi_h$ sending $s \in \Gamma$ to $\varphi_h(s)=gs$ and $sl\in \Gamma \times \Lambda$ to $\varphi_h(sl)=gsf(g)l$.\\
\\

An important example is the wreath product $(\mathbb{Z}/2\mathbb{Z}) \wr \mathbb{Z}$, often referred to as the \emph{lamplighter group}. It can be visualized as follows: consider a lamplighter located on an infinite street indexed by the integer. Along the street, there are lamps at each integer that can be either on or off. The lamplighter moves along the line, switching the state of the lamps as they go. The group $\mathbb{Z}$ corresponds to the position of the lamplighter, while the group $\bigoplus_{\mathbb{Z}} \mathbb{Z}/2\mathbb{Z}$ tracks the
configuration of the lamps. With this interpretation, we refer to $\Gamma$ as the base group and $\Lambda$ as the lamp group. A lamp configuration corresponds to a function $f \colon \Gamma \to \Lambda$ with finite support. Both the base group and the lamp group embed into the wreath product via the maps
\begin{equation}\label{eq:GammaLambdaEmbeds}
  \Gamma \ni g \hookrightarrow \hat{g}=(e,g), \qquad \Lambda \ni l \hookrightarrow \hat{l}=(1_{l},e_\Gamma),  
\end{equation}
where $e=e_{ {\underset{\Gamma}{\oplus}\Lambda}}$ and $1_{l}(g) = e_\Lambda$, if $g\not= e_\Gamma$ and $1_{l}(e_\Gamma) = l$ otherwise. We extend these embeddings to measures $\mathcal{M}_\dC(\Gamma), \mathcal{M}_\dC(\Lambda) \ni m \hookrightarrow \hat{m} \in \mathcal{M}_\dC(\Lambda \wr \Gamma).$

\paragraph{Main results.} The first explicit computation of the spectrum of wreath product goes back to \cite{grigorchuk2001lamplighter}, for the lamplighter group $(\dZ/2\dZ)\wr \dZ$ with suitable generators, leading to a first example of Cayley graph with dense but pure point spectrum, and to a counterexample to the Strong Atiyah Conjecture. Its connection to Random operator appeared \cite{lehner2008spectrum}, where lamplighters are related to percolation Hamiltonians. Our first theorem generalizes this correspondence to any wreath product, allowing to consider random operator with more general distribution than binary valued disorder.

\begin{theorem}\label{thm:wrcoresp}
Let $A=L_{m_\Gamma}$, $B=L_{m_\Lambda}$ be two self-adjoint convolution operators on countable groups $\Gamma,\Lambda$. Consider the random Schr\"{o}dinger operator on of the form $H(\omega)=A+V$ such that the random potential $(v(x,\omega))_{x\in \Gamma}$ is i.i.d, following the spectral measure $\mu_B^{e_\Lambda}$. On the wreath product $\Lambda\wr \Gamma$, consider the convolution operator $M=L_{m_{\Lambda \wr \Gamma}}$ with $m_{\Lambda \wr \Gamma}=\hat{m}_\Gamma+\hat{m}_\Lambda$. Then
\begin{equation}\label{eq:wrcoresp}
      \dE \mu_{H(\omega)}^{e_{\Gamma}}=\mu_M^{e_{\Lambda \wr \Gamma}}.
   \end{equation}
\end{theorem}
In the above proposition, the measure $m_{\Lambda \wr \Gamma}$ of the convolution operator $M$ is supported on the union of the embeddings $\hat{\Gamma}$ and $\hat{\Lambda}$. Regarding the lamplighter interpretation, the convolution corresponds to \textit{switch or walk} action. One can adapt the method to more general measures, such as the \textit{switch-walk-switch} random walk considered in \cite{lehner2008spectrum,grigorchuk2001lamplighter}. The averaged spectral measure $\bar{\mu}_H := \dE \mu_{H(\omega)}^{e_{\Gamma}}$ that appears in \eqref{eq:wrcoresp} is know as the \textit{Density of States} (DOS) in physics literature \cite{veselic2008existence,elek20072}. By ergodicity, this measure can alternatively be defined at a fixed realization $H(\omega)$, through a spatial average
$$\bar{\mu}_{H} \overset{a.s.}{=} \lim_{\Gamma_n \nearrow \Gamma}\frac{1}{|\Gamma_n|}\sum_{x\in \Gamma_n}\mu_{H(\omega)}^x.$$
This result holds under mild assumptions, for an increasing spanning sequence of subsets $\Gamma_n$ and the weak topology of measures. The DOS measure can be thought of as a generalization of the empirical spectral distribution, $\bar{\mu}_H = n^{-1} \sum_{i=1}^n \delta_{\lambda_i}$, of a finite-size matrix, and is therefore a purely spectral quantity. A natural question to consider is whether the identity \eqref{eq:wrcoresp} can be extended to a relation involving eigenvectors. This leads to our second main theorem, where we restrict ourselves to an Abelian group of lamps $\Lambda$. 

\begin{theorem}\label{thm:abeliancoresp} Let $A=L_{m_\Gamma}$, $B=L_{m_\Lambda}$ and $M=L_{m_{\Lambda \wr \Gamma}}$ as in theorem \ref{thm:wrcoresp} and suppose that $\Lambda$ is a finitely generated Abelian group. There is a natural construction of the probability space $(\Omega,\dP)$ on which $H(\omega)=A+V(\omega)$ is defined, such that $M$ is unitary equivalent to the direct integral of $H(\omega)$ over $(\Omega,\dP)$. More precisely, for some explicit unitary transform $\lF \colon L^2(\Omega, \dP; \ell^2(\Gamma)) \to \ell^2(\Lambda \wr \Gamma)$, one has
\begin{equation}\label{eq:UnitaryEquiv}
  \int_{\Omega}^\oplus H(\omega)\dP(\dd \omega) = \lF^* M \lF.  
\end{equation}
\end{theorem}
In the above theorem, the probability space $(\Omega,\dP)$ will correspond to the Pontryagin dual of the Abelian group of lamp configuration. Some consequence of the Plancherel formula will be discussed in the last section.

\begin{remark} In this work, our main motivation is to highlight a link between operators arsing in two different context, namely mathematical physics and algebra. We have therefore chosen to state the main results in their simplest forms. For instance, we only consider bounded potential, to avoid the technicalities of unbounded self-adjoint operators. Furthermore, we have presented our results in the fully transitive case, where the Laplacian acts by convolution and the disorder distribution (resp. the lamp group) remain the same at each site. Most of our results generalize naturally to any operator $H(\omega)=A+V(\omega)$ where $A$ acting on an space $\ell^2(\lX)$, and where the distribution of the potential (resp. the group of lamp) may vary at each site $x\in \lX$. This would correspond to the setting of the generalized wreath product $(\Lambda_x)_{x\in \Gamma} \wr \Gamma$ considered in \cite{erschler2006isoperimetry}. 
\end{remark}

\section{Applications}

\subsection{Absolutely continuous spectrum on wreath product group}

As pointed out in the introduction, proving the (almost sure) presence of an absolutely continuous component for the spectral measures of a random Schr\"{o}dinger operator is a notably difficult problem. Similarly, to the best of the author’s knowledge, there is no general method available for determining the absolute continuity of the Plancherel measure $\mu_{L_m}^x$ of a (deterministic) convolution operator. However, in the special case where the measure satisfies certain commutation relations with additive characters, a criterion has been derived in \cite{muantoiu2007spectral}. As an immediate consequence of Theorem \ref{thm:wrcoresp}, we obtain the following result regarding convolutions on wreath products.

\begin{theorem}\label{thm:ACspecForWreath}
Consider a convolution operator $M = L_{m_{\Lambda \wr \Gamma}}$ on a wreath product $\Lambda \wr \Gamma$. Suppose that the measure $m_{\Lambda \wr \Gamma}$ decomposes, as in Theorem \ref{thm:wrcoresp}, into a sum $\hat{m}_\Lambda + \hat{m}_\Gamma$. Assume that the spectral measure $\mu_{M_{\Lambda}}^{e_\Lambda}$ is absolutely continuous with bounded density. Then, the same holds for the Plancherel measure $\mu_M^{e_{\Lambda \wr \Gamma}}$, with the same bound. In particular, $M$ is a purely absolutely continuous operator.
\end{theorem}

\begin{example} Consider $\Lambda$ as the free group $\mathbb{F}_d, d\geq2$ or the free Abelian group $\dZ^d$, $d\geq3$, with standard set of generator $T$. For any finitely generated group $\Gamma = \langle S \rangle$, the Cayley graph on $\Lambda \wr \Gamma$ with generator $\hat{S} \sqcup \hat{T}$ has absolutely continuous spectrum. 
\end{example}

\begin{proof}[Proof of theorem \ref{thm:ACspecForWreath}] This follows from the correspondence \eqref{eq:wrcoresp} together with a regularity estimate for the density of states of Schr\"{o}dinger operator, called Wegner estimate (e.g. \cite[Chapter 4]{aizenman2015random}). Consider the random Schr\"{o}dinger operator $H(\omega)=A+\operatorname{Diag}(v_x)_{x\in \Gamma}$, conditionally to the value of potentials $\{v_y\}_{x\not= e_\Gamma}$, we obtain the rank one perturbation family: 
$$H_0 = A + \sum_{y \in \Gamma \setminus \{e_\Gamma\}} v_y \delta_y \delta_y^*,\qquad H_v = H_0 + v \delta_{e_\Gamma}\delta_{e_\Gamma}^*,\quad v\in \dR.$$
Given $z\in \dC \setminus \Sigma(H_v)$, the diagonal element $g_v(z) = \langle \delta_{e_\Gamma}, (H_v-z)^{-1}\delta_{e_\Gamma}\rangle$ of the resolvent coincides with the Stieltjes transform of the spectral measure $\mu_{H_v}^{e_{\Gamma}}$. The second resolvent identity, $(H_v-z)^{-1}-(H_0-z)^{-1}=-(H_v-z)^{-1}v\delta_{e_\Gamma}\delta_{e_\Gamma}^*(H_0-z)^{-1}$, allows to express $g_v(z)$ explicitly in terms of $g_0$ and $v$
$$g_v(z)=\frac{1}{g_0(z)^{-1}+v}.$$ 
For $z$ with positive imaginary part, we obtain by averaging over the distribution $\mu_B^{e_\Lambda}(\dd v )=\rho(v) \dd v$,  
$$\frac{1}{\pi}\int \Im g_v(z) \rho(v)\dd v \leq \frac{\|\rho\|_\infty}{\pi} \int_\dR \Im g_v(z) \dd v =  \frac{\|\rho\|_\infty}{\pi}\int_\dR \Im \frac{1}{g_0(z)^{-1}+v} \dd v = \|\rho\|_\infty.$$
Since this bound holds independently of $\{v_y\}_{x\not= e_\Gamma}$, it comes
$$\frac{1}{\pi}  \Im  \int \frac{\bar{\mu}_{H}(\dd \lambda)}{\lambda-z} = \frac{1}{\pi}\dE \Im \langle \delta_{e_\Gamma}, (H(\omega) -z)^{-1}\delta_{e_\Gamma}\rangle \leq \|\rho\|_\infty.$$
As the imaginary part of the Stieltjes transform is the convolution with the Cauchy kernel, one recovers the measure when the complex parameter $z$ approaches the real axis. From \cite[Lemma 4.2]{aizenman2015random}, the measure $\bar{\mu}_H$ is absolutely continuous with density bounded by $\| \rho \|_{\infty}$. This complete the proof, since $\mu_{M}^{e_{\Lambda \wr \Gamma}} = \bar{\mu}_{H}$.
\end{proof}

\subsection{Lifchitz tail on polynomial growth group}

Now, we  illustrate our correspondence in the other direction. From known estimate on wreath product, we will quantify the decay of the density of states near the spectral edges. Such estimate goes back to Lifshitz \cite{lifshitz1965energy} who had the following intuition. Consider a random Schrodinger operator $H(\omega)=A+V$, with the standard adjacency operator $\dZ^d$ and bounded potential $v_{\max}\overset{a.s.}{=} \sup_x v(x,\omega)$. Suppose that the Rayleigh quotient of some unit vector $\varphi$ is almost optimal 
$$\langle \varphi, H(\omega) \varphi \rangle \geq \max \Sigma(H)-\varepsilon, \qquad \varepsilon \ll 1.$$ 
Since $\max \Sigma(H)=2d+v_{\max}$, it follows by the triangle inequality and that a similar relation holds when one replaces $H$ either by $A$ or by $V(\omega)$. On the one hand, $\langle \varphi, A \varphi \rangle \geq 2d - \varepsilon$ implies by uncertainty principle that the vector has $\ell^2$ mass delocalized on a set of size $\varepsilon^{-d/2}$. On the other hand, $\langle \varphi, V(\omega) \varphi \rangle \geq v_{\max}-\varepsilon$ implies that the potential is nearly maximal on this set. By large deviation of independent events, this occurs with probability decaying exponentially in $\varepsilon^{-d/2}$. Since the density of state has the interpretation of average number of eigenvalues by unit volume, this suggests the following decay for the DOS measure at the spectral edge $\sigma :=\max \Sigma(H)$:
$$\bar{\mu}_{H}\left([\sigma-\varepsilon,\sigma] \right) \simeq \exp \left( \varepsilon^{-d/2}\right).$$
We refers to Theorem \ref{thm:LifshitzForPoly} below for a rigourous statement. Such relations, known as Lifshitz tails, are often used as an input for proving Anderson localization. An elementary proof for Euclidean lattices can be found in \cite{simon1985lifschitz}. A similar result for the $d$-regular tree was established in \cite{hoecker2014anderson}, with a double-exponential decay arises due to the exponential growth of $\dT_d$. Both proofs rely on the Fourier decomposition of the adjacency operator. Here, we extend this universal phenomenon to any polynomial growth group.\\

Recall that a non-negative measure $ m $ on a finitely generated group $ \Gamma = \langle S \rangle $ has finite second moment if $\sum_{x \in \Gamma} |x|_S^2 m(x) < \infty$, where $|x|_S$ denotes the distance to the identity $e_\Gamma$ in the associated Cayley graph. In this setting, the group is said to have polynomial growth if the size $V(n)$ of the balls of radius $ n $ in the Cayley graph satisfies $ V(n) \leq C n^d$, for some constants $ C, d > 0 $. A fundamental result of Gromov \cite{gromov1981groups} states that the smallest such $d$ is an integer independent of $S$, known as the growth degree of $\Gamma$.

\begin{theorem}[Lifshitz tail on any polynomial growth group]\label{thm:LifshitzForPoly}
Consider a Schr\"{o}dinger operator $H(\omega)=A+V$ such that the correspondence of Theorem \ref{thm:wrcoresp} holds. We assume that both groups $\Gamma,\Lambda$ have polynomial growth and that the measure $m_{\Lambda \wr \Gamma}$ is non-negative with finite second moment, containing a finite generating set. If $\Gamma$ is infinite,
$$\lim_{\varepsilon \downarrow 0} \frac{\ln (- \ln (\bar{\mu}_{H}\left([\sigma-\varepsilon,\sigma] \right) ))}{\ln \varepsilon}=-\frac{d}{2},$$
where $\sigma=\max \Sigma(H)$ and $d$ is the growth degree of $\Gamma$. 
\end{theorem}

\begin{example} Consider the adjacency operator $A$ the discrete Heisenberg with respect to the standard presentation $\Gamma=H_3(\dZ)=\langle x,y,z\, |\, [x,z]=[y,z]=e, [x,y]=z \rangle$, which has polynomial growth $d=4$. Let $V$ be i.i.d potential of value $\pm \sigma$. Then the operator $H=A+V$ can easily seen in correspondence a convolution operator on $(\dZ/2\dZ) \wr H_r(\dZ)$. Following remark \ref{rmk:VonZ}, this also holds for a large class of potential distribution, by considering the wreath product $\Lambda \wr \Gamma$, $\Lambda$ being a finitely generated Abelian group.
\end{example}

\begin{proof} By assumption, the density of state measure $\bar{\mu}_{H}=\dE \mu_{H(\omega)}^{e_\Gamma}$ coincides with the Plancherel measure of the convolution operator $M=L_{m_{\Lambda \wr \Gamma}}$. Up to a scaling, we can assume that $m_{\Lambda \wr \Gamma}$ is a probability measure. From Kesten criterion on amenable groups, it follows that $\sigma = \max \Sigma(M)=\| M \| = 1$. The moments $m_{2n}:=\int \lambda^{2n} \mu_M^{e_{\Lambda \wr \Gamma}}(\dd \lambda)$ correspond to the return probability of the simple random walk with iid increment following $m_{\Lambda \wr \Gamma}$. We write $f \prec g$ if one has $f(n) \leq C g(c n)$ for some constant $c,C>0$. If both $f\prec g$ and $g\prec f$, we write $ f\sim g$. It is well known that the asymptotic return probability does not depend on the choice of the measure, as soon as it is symmetric, have second moment, and generating the group, see \cite{pittet2000stability}. From a careful analysis of isoperimetric profile on wreath products, the asymptotic return probability has been computed in \cite[Theorem 2]{erschler2006isoperimetry}. One has
$$m_{2n} \sim \exp(-n^{d/d+2}\log^{\alpha}(n)),$$
where $\alpha = 2/(d+2)$ if $\Lambda$ is infinite, and $\alpha = 0$ otherwise. It remains to relate the behavior of the return probability to the spectral measure via heat kernel estimates. Let
$$k(t)=\langle \delta_{e_{\Lambda \wr \Gamma}}, e^{-t(I-M)} \delta_{e_{\Lambda \wr \Gamma}}\rangle = \int_{-1}^1e^{-t(1-\lambda) } \mu_{M}^{e_{\Lambda \wr \Gamma}}(\dd \lambda).$$
When $t$ and $n$ have the same order, the return probability and the heat kernel share the same asymptotic, e.g. from \cite[proposition 3.2]{pittet2000stability}, one has
$$m_{2n+2} \leq 2k(2n), \quad k(4n)\leq e^{-2n} + m_{2n}.$$
Let $I_\varepsilon = [1-\varepsilon,1]$, one has
$$e^{-\varepsilon t}\bar{\mu}_{H}\left( I_\varepsilon \right) \leq k(t) = \int_{-1}^1e^{-t(1-\lambda) } \bar{\mu}_{H}(\dd \lambda) \leq e^{-\varepsilon t} + \bar{\mu}_{H}(I_\varepsilon).$$
We start with the upper bound. For some $C,c>0$, we obtain if $t$ is large enough, 
$$\bar{\mu}_{H}(I_\varepsilon) \leq e^{\varepsilon t}(e^{-t/2}+Ce^{-c t^{d/{d+2}}\ln^\alpha(t)}) \leq 2Ce^{-ct^{d/d+2}+\varepsilon t}$$
Taking $\varepsilon t = ct^{d/d+2} /2$ and $\varepsilon\downarrow 0$, we get for some $c'>0$
$$\bar{\mu}_{H}(I_\varepsilon) \leq 2Ce^{-c \varepsilon t/2} = 2C \exp(-c'\varepsilon ^{-d/2})$$
from which $\limsup_{\varepsilon \downarrow 0} \ln(-\ln(\bar{\mu}_{H}(I_\varepsilon)))/\ln(\varepsilon) \leq -d/2$ follows. For the lower bound, one has for some other constant $C,c>0$
$$\bar{\mu}_{H}(I_\varepsilon) \geq -e^{\varepsilon t}+Ce^{-ct^{d/d+2}\ln^{\alpha}(t)} = Ce^{-ct^{d/d+2}\ln^{\alpha}(t)} (1-e^{-\varepsilon t + ct^{d/d+2}\ln^{\alpha}(t)}/C).$$
Taking $\varepsilon t = 2ct^{d/d+2}\ln^{\alpha}(t)=2ct^{\frac{d}{d+2}+o(1)}$, $\varepsilon\downarrow 0$, we get 
$$\bar{\mu}_{H}(I_\varepsilon) \geq \tfrac{C}{2} e^{-ct^{d/d+2}\ln^{\alpha}(t)} \geq C/2\exp(\varepsilon^{\frac{-d}{2}+o(1)}).$$
and the result follows.
\end{proof}

\subsection{An exact formula for the Green’s function second moment}

Concerning random Schr\"{o}dinger operators, one is often more interested in the almost sure properties of individual realization rather than averaged quantity such as the density of states. The following proposition formally asserts that $H(\omega)$ can be reconstructed as a Fourier series in terms of the operator $M$. Consider the setting of Theorem \ref{thm:abeliancoresp}: a random Schrodinger operator $H(\omega)$ whose direct integral $\int_\Omega^\oplus H(\omega) \dP( \dd \omega)$ is unitary equivalent to a convolution operator $M=L_{m_{\Lambda \wr \Gamma}}$, and $\Lambda$ is a finitely generated Abelian group. For simplicity, we assume $\Lambda = \dZ$.
\begin{proposition}[Reverse correspondence \label{prop:reversecoresp}] For any smooth bounded function $f\colon \Sigma(M) \to \dC$, the following identity holds in $L^2(\Omega,\dP)$
\begin{equation}\label{eq:HfourierL}
  \langle \delta_{e_\Gamma} , f(H(\omega)) \delta_g \rangle = \sum_{l\in \bigoplus_\Gamma \dZ} \langle \delta_{e_{\dZ \wr \Gamma}}, f(M) \delta_{(l,g)}\rangle e^{2\pi i l\cdot \omega}.  
\end{equation}
In particular, Parseval identity applies
\begin{equation}\label{eq:Parseval}
  \dE |\langle \delta_{e_\Gamma}, f(H(\omega)) \delta_g \rangle |^2 = \sum_{l\in \bigoplus_\Gamma \dZ} |\langle \delta_{e_{\dZ \wr \Gamma}}, f(M) \delta_{(l,g)}\rangle |^2. 
\end{equation}
\end{proposition}
\begin{remark}
    Setting $ f(\lambda) = (\lambda - z)^{-1} $, where $ z \in \mathbb{C} \setminus \mathbb{R} $, equation \eqref{eq:Parseval} provides an exact formula for the second moment of the complex Green’s function. For random Schrödinger operators, one possible approach for proving the almost sure absolute continuity of the spectrum is to show that this quantity remains bounded as $ z $ approaches the real axis; see \cite[Lemma 20]{arras2023existence}. The above proposition has a surprising interpretation. The Green’s function admits a second moment, given by $\mathbb{E} \left| \langle \delta_{e_\Gamma}, (H(\omega)-z)^{-1} \delta_{e_\Gamma} \rangle \right|^2$, if the quantity $\sum_l \left| \langle \delta_{e_{\mathbb{Z} \wr \Gamma}}, (M-z)^{-1} \delta_{(e_{\Gamma},g)}\rangle \right|^2$ remains finite. This implies that the eigenfunctions of $ M $ decay sufficiently fast in $ l $. Consequently, the delocalization of $ H(\omega) $ is related to the localization properties of the eigenfunctions of $ M $ along the group of lamps.
\end{remark}

\begin{proof}[Proof of proposition \ref{prop:reversecoresp}.] It suffices to compute the Fourier coefficient of the map $\omega \mapsto \langle \delta_e , f(H(\omega)) \delta_g \rangle$, seen as an element of the Hilbert space $L^2(\Omega,\dP)$. From the basic theory of direct integral \cite{dixmier1969c}, one has for a smooth function $f$ 
$$f\left(\int_{\Omega}^{\oplus}H(\omega) \dd \dP(\omega)\right)=\int_{\Omega}^{\oplus}f(H(\omega)) \dd \dP(\omega),$$
and from theorem \ref{thm:abeliancoresp}
$$   \int_{\Omega} \langle \delta_e, f(H(\omega)) \delta_g \rangle \chi_l(\omega)\dd \dP(\omega)= \langle \chi_0\delta_{e_\Gamma}, \left(\int_{\Omega}^{\oplus}f(H(\omega)) \dd \dP(\omega)\right)  \chi_l \delta_x \rangle = \langle \delta_{e_{\dZ \wr \Gamma}}, f(M) \delta_{(l,g)}\rangle.$$
\end{proof}

\section{Proofs of main theorems}

\begin{proof}[Proof of theorem \ref{thm:wrcoresp}]
All operators are bounded in norm by the sum of the absolute mass of the measures $m_\Gamma$ and $m_\Lambda$. The identity \eqref{eq:wrcoresp} reduces to the equalities of the sequence of moments
\begin{equation}\label{eq:momentmatch}
    \dE \langle \delta_{e_\Gamma}, H(\omega)^n \delta_{e_\Gamma} \rangle = \langle \delta_{e_{\Lambda \wr \Gamma}}, M^n \delta_{e_{\Lambda \wr \Gamma}} \rangle, \quad n\in \dN.
\end{equation}
Given a new formal variable $v$, we consider the set $W_n$ of all words of length $n$ in the alphabet $\Gamma \sqcup \{v\}$ which reduces to the identity element $e_\Gamma$
$$W_n = \left\{u=u_1u_2 \cdots u_n\in 
 (\Gamma \sqcup \{v\})^n \middle| \prod_{\substack{i=1\\u_i \not=v}}^n u_i = e_\Gamma \right\}.$$
We now express both sides of \eqref{eq:momentmatch} in terms of a summation over $W_n$. It is clear from the group multiplication that any word $u$ gives rise to a closed walk $x_0,x_1,\ldots x_n$ in $\Gamma$. Starting from $x_o=e_\Gamma$, one sets $x_i = u_i^{-1} x_{i-1}$ if $u_i \in \Gamma$, and $x_i = x_{i-1}$ if $u_i = v$.  To obtain the expansion $\langle \delta_{e_\Gamma}, H(\omega)^n \delta_{e_\Gamma} \rangle$, one must keep track of the weights of the walk. We assign $m_\Gamma(u_i)$ if $u_i \in \Gamma$, otherwise $u_i=v$ and we assign the value $v_{x_i}=v({x_i},\omega)$ of the potential at the current position $x_i=x_{i-1}$ of the walk. This gives
$$\langle \delta_{e_\Gamma}, H(\omega)^n \delta_{e_\Gamma} \rangle = \sum_{u \in W_n} \prod_{\substack{i=1\\u_i \in \Gamma}}^n m_\Gamma(u_i)\prod_{\substack{j=1\\u_j = v}}^n v_{x_j}.$$
In order to use the independence of the $\{v({x_i},\omega)\}_{x\in \Gamma}$, it is convenient to reorganize the sum. For each $y\in \Gamma$, let $k(u)$ denote the number of occurrences of $v_y$ in the last product:
$$k_y(u)=\# \left\{i\in \{1,\ldots ,n\}\middle| u_i = v \textmd{ and } x_i = y \right\},\qquad y \in \Gamma.$$
Note that for most $y$ one has simply $k_y(u)=0$. We obtain
$$\dE \langle \delta_{e_\Gamma}, H(\omega)^n \delta_{e_\Gamma} \rangle =  \sum_{u \in W_n} \prod_{\substack{i=1\\u_i \in \Gamma}}^n m_\Gamma(u_i)\prod_{y \in \Gamma} \dE[v_y^{k_y(u)}].$$
We now express the left hand side of \eqref{eq:momentmatch}. By construction, the action of $M$ on any element $(f,g)\in \Lambda \wr \Gamma$ is given by
$$M\delta_{(f,g)}=\sum_{s \in \Gamma} m_\Gamma(s)\delta_{\hat{s}^{-1}(f,g)}+\sum_{t \in \Lambda} m_\Lambda(t)\delta_{\hat{t}^{-1}(f,g)}.$$
We follow the lamplighter interpretation: the first sum corresponds to displacements of the lamplighter from $g$ to $s^{-1}g$. The second sum corresponds to a switch of the lamp at the position $g$ of the lamplighter, from the lamp state $f(g)$ to $t^{-1}f(g)$. A given word $u\in W_k$ already prescribes the displacement of the lamplighter. A letter $u_i =v$ with $x_i=y$ will correspond to an update of the lamp $y$. In total, there are $k_y(u)$ updates of the lamp at $y$. One must therefore choose a sequence $t_y=t^y_1\cdots t^y_{k_y(u)} \in \Lambda^{k(u)}$ to obtain a walk from $e_{\Lambda \wr \Gamma}$ with $n$ steps. For the walk in $\Lambda \wr \Gamma$ to be closed, each lamp must be in the state $e_\Lambda$ at the end of the walk. We introduce $T_k$, the subset of $\Lambda^k$ which reduces to the identity $e_\Lambda$ by the group multiplication:
$$T_k = \left\{t=t_1 \cdots t_k \in \Lambda^k \middle|\prod_{i=1}^k t_i =e_\Lambda \right\}.$$
Keeping track of the weight, we obtain:
$$\langle \delta_{e_{\Lambda \wr \Gamma}}, M^n \delta_{e_{\Lambda \wr \Gamma}} \rangle = \sum_{u \in W_n, } \prod_{\substack{i=1\\u_i \in \Gamma}}^n m_\Gamma(u_i)\prod_{y \in \Gamma} \left(\sum_{t^y \in T_{k_y(u)}}\prod_{j=1}^{k_y(u)}m_\Lambda(t_j^y) \right).$$ 
To conclude, we recall that the one-site potential distribution follows the spectral measure of the convolution operator $B=L_{m_\Lambda}$. Since
$\int v^k \mu_B^{e_\lambda}(\dd v) =  \langle \delta_{e_\Lambda}, B^k \delta_{e_\Lambda} \rangle = \sum_{t \in T_{k}}\prod_{j=1}^{k}m_\Lambda(t_j)$, it follows
$$\langle \delta_{e_{\Lambda \wr \Gamma}}, M^n \delta_{e_{\Lambda \wr \Gamma}} \rangle = \sum_{u \in W_n, } \prod_{\substack{i=1\\u_i \in \Gamma}}^n m_\Gamma(u_i)\prod_{y \in \Gamma} \left( \dE [v_y^{k_y(u)}]\right)=\dE \langle \delta_{e_\Gamma}, H(\omega)^n \delta_{e_\Gamma} \rangle.$$
\end{proof}


When the group of lamp is $\Lambda = \dZ/2\dZ$, and the base group $\Gamma$ is finite, the wreath product has order $|\Gamma|2^{|\Gamma|}$, matching the dimension of the matrix $H_\Omega=\oplus_{\omega \in \Omega} H(\omega)$, where in the direct sum ranges over the $2^{|\Gamma|}$ possible assignments of the potential. From theorem \ref{thm:wrcoresp}, it follows that $M$ is unitary equivalent to the direct sum $H_\Omega$. This suggests that the correspondence can be extended to eigenvectors. We first gather some basic facts about the concept of direct integral, which generalizes the direct sum. A comprehensive treatment can be found in \cite{dixmier1969c}.

\paragraph{Direct integral.}We will only be concerned with direct integrals of a constant Hilbert field, which allows us to avoid many technicalities concerning measurability. Given a Hilbert space $\lH$ and a standard Borel space $(X,\mu)$, we denote $\lH_X$ the direct integral of $\lH$ over $(X,\mu)$,
$$\lH_X = \int_{X}^{\oplus} \lH \,\mu(\dd s):=L^2(X,\mu; \lH).$$
This is the space of (equivalent classes of) weakly measurable maps from $X$ to $\lH$, verifying the summability condition
$$\|\varphi\|_{\lH_X}^2=\int_{X} \|\varphi(s)\|_{\lH}^2  \mu(  \dd s).$$
A uniformly bounded operator valued map $A \colon X \to \lB(\lH)$ can be turn into an operator $A_X$, denoted
$$A_X := \int_{X}^{\oplus}A(s)\mu( \dd s),$$
This is the direct integral of $A(s)$ over $X$ and is defined point-wise by $(A_X\varphi)(s)=A(s)\varphi(s)$.

\begin{remark}
For non-constant fields $s\mapsto \lH_s$ of Hilbert spaces, the construction of $\lH_X = \int_{X}^{\oplus} \lH_s \mu( \dd s)$ is similar but depends on a choice of a basis for each space $\lH_s$ that is compatible with the structure $(X,\mu)$, see \cite{dixmier1969c}. An abstract formulation of the spectral theorem says that any self-adjoint operator is unitary equivalent to the direct integral of a real multiple of the identity operator: for some Borel space $(X,\mu)$, there is a measurable field of Hilbert space $(\lH_s)$ together with a real map $\lambda(s)$ such that $H \simeq \int_X^{\oplus}\lambda(s)I_{\lH_s}\mu( \dd s)$. We refer to \cite[Theorem 7.19]{hall2013quantum} for more details.    
\end{remark}

As a first step toward theorem \ref{thm:abeliancoresp}, we consider the Abelian group $\Lambda = \dZ$ and three operators $A=L_{m_\Gamma}$, $B=L_{m_\dZ}$, $M=L_{\hat{m}_\Gamma+\hat{m}_\dZ}$ acting respectively on $\ell^2(\Gamma),\ell^2(\dZ)$ and $\ell^2(\dZ \wr \Gamma)$. 

\begin{lemma}\label{lmm:Vdistribution} Let $V_{m_\dZ}(\cdot)=\sum_n m_\dZ(n)e^{i2\pi n \cdot} \in L^2(\dR/\dZ)$ be the Fourier transform of the symmetric measure $m_\dZ \in \mathcal{M}_{\dC}(\dZ)$. If $U$ follows the uniform distribution on $[0,1]$, then $V_{m_\dZ}(U)$ is distributed according to the spectral measure $\mu_{B}^{e_\dZ}$.
\end{lemma}

\begin{remark}\label{rmk:VonZ} Conversely, most potential distribution $\nu$ can be realized in this manner. Simply take $F^{-1}(U)$, the generalized inverse of the inverse cumulative distribution $F(t)=\nu((-\infty,t])$. However, it is not guaranteed that the Fourier series $F^{-1}(\cdot)=\sum_n m(n)e^{i 2\pi n \cdot}$ is in $\ell^1(\dZ)$.
\end{remark}

\begin{proof}[Proof lemma \ref{lmm:Vdistribution}] Since $m_\dZ$ is symmetric and $\ell^1$, $V_{m_\dZ}(\cdot)=\sum_{n} m_\dZ(n)e^{2\pi i \cdot}$ is a well-defined real and bounded function. The fact that $V_{m_\dZ}(U)$ is distributed following $\mu_{B}^{e_\dZ}$ follows from $\int_{[0,1]}V_{m_\dZ}(u)^k \dd u = \langle \delta_{e_\dZ}, B^k \delta_{e_\dZ} \rangle$, $k\in \dN$.
\end{proof}
Consider the probability space $\Omega=[0,1]^\Gamma$ with the cylinder algebra and $\dP$ the uniform product measure. On this space, we construct $H(\omega)=A+V$ as the operator valued random variable with
\begin{equation}\label{eq:defHomega}
  V=\operatorname{Diag}(v(x,\omega))_{x\in \Gamma}, \quad v(x,\omega)=V_{m_\dZ}(\omega_x), \quad \omega=( \omega_x ) \in [0,1]^\Gamma.  
\end{equation} 
We recall that the Fourier transform from $\dZ$ to $\dR/\dZ \simeq [0,1]$ extends to countable product (e.g \cite[chapter 4]{folland2016course}). We denote $\lF_1$ the induced unitary map
\begin{equation}\label{eq:FourierDef}
  \lF_1 \colon \ell^2\Big(\bigoplus_{\Gamma} \dZ \Big) \to L^2([0,1]^\Gamma, \dP),  
\end{equation}
which sends any basis element $\delta_f, f\in \bigoplus_\Gamma \dZ$ to the character function $\chi_f$ defined as the product $\chi_f(\omega)=e^{i2\pi f\cdot \omega}=\prod_{x\in\Gamma}e^{i2\pi f(x)\omega_x}$.

\begin{theorem}[Unitary equivalence]\label{thm:Zcorresp} Let $M=L_{m_{\dZ \wr \Gamma}}$ with $m_{\dZ \wr \Gamma}=\hat{m}_\Gamma+\hat{m}_\dZ$ as in theorem \ref{thm:wrcoresp}, and let $H(\omega)$ the operator valued random variable on the probability space $([0,1]^\Gamma,\dP)$ as constructed in \eqref{eq:defHomega}. Then $M$ is unitary equivalent to the direct integral of $H(\omega)$ over $(\Omega,\dP)$. More precisely if $\lF=\lF_1 \otimes I_{\ell^2(\Gamma)}$. One has
\begin{equation}\label{eq:unitaryequiv}
    H_\Omega = \int_{\Omega}^{\oplus}H(\omega) \dP( \dd \omega) = \lF M \lF^*.
\end{equation}
\end{theorem}

\begin{proof}[Proof of theorem \ref{thm:Zcorresp}.] By construction, $\lF$ maps the canonical basis $\delta_{h}$, $h=(f,g)$ to the unit element of $L^2(\Omega,\dP,\ell^2(\Gamma))$ defined as $\chi_f\delta_g \colon \omega \mapsto \chi_f(\omega)\delta_g \in \ell^2(\Gamma)$. Equation \eqref{eq:unitaryequiv} follows from a computation on the canonical basis. On the one hand we have, for $h' = (f',g')$,
\begin{align*}
    \langle \delta_{h'},M \delta_{h}\rangle &= \sum_{s \in \Gamma} m_\Gamma(s)\langle  \delta_{(f',g')}, \delta_{\hat{s} (f,g)} \rangle +\sum_{t \in \dZ} m_\Lambda(t)\langle  \delta_{(l',g')}, \delta_{\hat{t} (f,g)} \rangle  \\
    &= 1_{f=f'} \langle \delta_g'|A \delta_g\rangle + 1_{\{g=g'\}}\left(\prod_{\substack{x\in \Gamma\\
    x\not=g}}1_{\{l_x =l'_x\}}\right) \langle \delta_{f'(g)},B \delta_{f(g)}\rangle,
\intertext{and on the other hand we have the direct integral }
    \langle \chi_{l'} \delta_{g'}, H_\Omega \chi_l \delta_g \rangle &= \int_{\Omega} \langle \chi_{l'} (\omega) \delta_{g'}, (H\chi_l\delta_g)(\omega) \rangle  \dd \dP(\omega)\\
    &= \int_{\Omega} \langle \delta_{g'} , A \delta_g + v_g(\omega)\delta_g \rangle \overline{\chi_{l'}} (\omega)\chi_l(\omega)\dd \dP(\omega)\\
    &= 1_{l=l'} \langle \delta_g'|A \delta_g\rangle + 1_{\{g=g'\}} \int_{\Omega}V(\omega_x)e^{2\pi i(l-l')\cdot \omega}\dd\dP(\omega)\\
    &= 1_{l=l'} \langle \delta_g'|A \delta_g\rangle + 1_{\{g=g'\}} \left(\prod_{\substack{x\in \Gamma\\
    x\not=g
    }}1_{\{l_x =l'_x\}}\right)  \langle \delta_{l_g},B \delta_{l_g}\rangle.
\end{align*}
\end{proof}

The proof generalizes easily for any finitely generated Abelian group.

\begin{proof}[Proof of Theorem \ref{thm:abeliancoresp}]
By assumption, the group of lamps is a finitely generated Abelian group. Hence, it can be expressed as $\Lambda \simeq \mathbb{Z}^n \oplus \mathbb{Z}/d_1\mathbb{Z} \oplus \cdots \oplus \mathbb{Z}/d_k\mathbb{Z}$. Its dual group is the compact Abelian group $\hat{\Lambda} \simeq [0,1]^n \oplus \mathbb{Z}/d_1\mathbb{Z} \oplus \cdots \oplus \mathbb{Z}/d_k\mathbb{Z}$, endowed with the normalized Haar measure denoted $\mu$. Under the Fourier transform, the convolution operator $B = L_{m_{\Lambda}}$ becomes a multiplication operator $L^2(\hat{\Lambda},\mu)$, by a real function denoted $V_{m_{\Lambda}}$. We define the probability space $(\Omega, \dP)$ as the Pontryagin dual of the group of lamp configurations 
$$ \Omega = \widehat{\bigoplus_{\Gamma} \Lambda} = \prod_{\Gamma} \hat{\Lambda},\qquad \dP = \otimes_{\Gamma} \mu.$$
For $\omega = (\omega_x)_{x \in \Gamma} \in \Omega$, we define  
$$
V(\omega) = \operatorname{Diag}(v(x, \omega)),  \qquad v(x, \omega) = V_{m_{\Lambda}}(\omega_x).$$
As in lemma \ref{lmm:Vdistribution}, $(v(x,\omega))$ is an independent and identically distributed potential, following the distribution $\mu_{B}^{e_\Lambda}$. Next, we construct the unitary transform $\mathcal{F} = \mathcal{F}_1 \otimes I_{\ell^2(\Gamma)}$, where $I_{\ell^2(\Gamma)}$ denotes the identity operator on $\ell^2(\Gamma)$ and $\mathcal{F}_1$ denotes the Fourier transform, which provides a unitary map from $\ell^2(\bigoplus_{\Gamma} \Lambda)$ to $L^2(\Omega, \dP)$. With this setup, the proof that 
$$ \int_{\Omega}^{\oplus} H(\omega) \, \dP(\dd \omega) = \mathcal{F} M \mathcal{F}^* $$
follows the same steps as for the case $\Lambda = \mathbb{Z}$ discussed earlier.
\end{proof}

\section*{Acknowledgments}
The author thanks Charles Bordenave and Christophe Pittet for helpful discussions. The author was partially supported by the KKP 139502 project and the ERC Dynasnet 810115 project. 

\bibliographystyle{plain}

\begin{thebibliography}{10}

\bibitem{aizenman1993localization}
Michael Aizenman and Stanislav Molchanov.
\newblock Localization at large disorder and at extreme energies: An elementary
  derivations.
\newblock {\em Communications in Mathematical Physics}, 157:245--278, 1993.

\bibitem{aizenman2006stability}
Michael Aizenman, Robert Sims, and Simone Warzel.
\newblock Stability of the absolutely continuous spectrum of random
  schr{\"o}dinger operators on tree graphs.
\newblock {\em Probability theory and related fields}, 136:363--394, 2006.

\bibitem{aizenman2015random}
Michael Aizenman and Simone Warzel.
\newblock {\em Random operators}, volume 168.
\newblock American Mathematical Soc., 2015.

\bibitem{anderson1958absence}
Philip~W Anderson.
\newblock Absence of diffusion in certain random lattices.
\newblock {\em Physical review}, 109(5):1492, 1958.

\bibitem{arras2023existence}
Adam Arras and Charles Bordenave.
\newblock Existence of absolutely continuous spectrum for galton--watson random
  trees.
\newblock {\em Communications in Mathematical Physics}, 403(1):495--527, 2023.

\bibitem{bordenave2016spectrum}
Charles Bordenave.
\newblock Spectrum of random graphs.
\newblock {\em Advanced topics in random matrices}, 53:91--150, 2016.

\bibitem{bourgain2003random}
J~Bourgain.
\newblock Random lattice schr{\"o}dinger operators with decaying potential:
  some higher dimensional phenomena.
\newblock In {\em Geometric Aspects of Functional Analysis: Israel Seminar
  2001-2002}, pages 70--98. Springer, 2003.

\bibitem{bourgain2001random}
Jean Bourgain.
\newblock On random schr{\"o}dinger operators on $\mathbb{Z}^2$.
\newblock {\em Discrete and Continuous Dynamical Systems}, 8(1):1--15, 2001.

\bibitem{de2024spectra}
Pierre de~la Harpe.
\newblock Spectra of infinite cayley graphs, examples with pure band spectra.
\newblock {\em arXiv preprint arXiv:2402.06279}, 2024.

\bibitem{dixmier1969c}
Jacques Dixmier.
\newblock Les c [*]-alg{\'e}bres et leurs repr{\'e}sentations.
\newblock {\em (No Title)}, 1969.

\bibitem{elek2003analytic}
G{\'a}bor Elek.
\newblock On the analytic zero divisor conjecture of linnell.
\newblock {\em Bulletin of the London Mathematical Society}, 35(2):236--238,
  2003.

\bibitem{elek20072}
G{\'a}bor Elek.
\newblock L\^{} 2-spectral invariants and convergent sequences of finite
  graphs.
\newblock {\em arXiv preprint arXiv:0709.1261}, 2007.

\bibitem{erschler2006isoperimetry}
Anna Erschler.
\newblock Isoperimetry for wreath products of markov chains and multiplicity of
  selfintersections of random walks.
\newblock {\em Probability Theory and related fields}, 136:560--586, 2006.

\bibitem{folland2016course}
Gerald~B Folland.
\newblock {\em A course in abstract harmonic analysis}.
\newblock CRC press, 2016.

\bibitem{froese2007absolutely}
Richard Froese, David Hasler, and Wolfgang Spitzer.
\newblock Absolutely continuous spectrum for the anderson model on a tree: a
  geometric proof of klein’s theorem.
\newblock {\em Communications in mathematical physics}, 269:239--257, 2007.

\bibitem{frohlich1983absence}
J{\"u}rg Fr{\"o}hlich and Thomas Spencer.
\newblock Absence of diffusion in the anderson tight binding model for large
  disorder or low energy.
\newblock {\em Communications in Mathematical Physics}, 88(2):151--184, 1983.

\bibitem{graf1994anderson}
Gian~Michele Graf.
\newblock Anderson localization and the space-time characteristic of continuum
  states.
\newblock {\em Journal of statistical physics}, 75:337--346, 1994.

\bibitem{grigorchuk2021laplace}
Rostislav Grigorchuk and Christophe Pittet.
\newblock Laplace and schr$\backslash$" odinger operators without eigenvalues
  on homogeneous amenable graphs.
\newblock {\em arXiv preprint arXiv:2102.13542}, 2021.

\bibitem{grigorchuk2001lamplighter}
Rostislav~I Grigorchuk and Andrzej {\.Z}uk.
\newblock The lamplighter group as a group generated by a 2-state automaton,
  and its spectrum.
\newblock {\em Geometriae Dedicata}, 87(1):209--244, 2001.

\bibitem{grigorchuk2004ihara}
Rostislav~I Grigorchuk and Andrzej Zuk.
\newblock The ihara zeta function of infinite graphs, the kns spectral measure
  and integrable maps.
\newblock {\em Random walks and geometry}, pages 141--180, 2004.

\bibitem{gromov1981groups}
Michael Gromov.
\newblock Groups of polynomial growth and expanding maps (with an appendix by
  jacques tits).
\newblock {\em Publications Math{\'e}matiques de l'IH{\'E}S}, 53:53--78, 1981.

\bibitem{hall2013quantum}
Brian~C Hall.
\newblock {\em Quantum theory for mathematicians}, volume 267.
\newblock Springer Science \& Business Media, 2013.

\bibitem{hoecker2014anderson}
Francisco Hoecker-Escuti and Christoph Schumacher.
\newblock The anderson model on the bethe lattice: Lifshitz tails.
\newblock 2014.

\bibitem{keller2012absolutely}
Matthias Keller, Daniel Lenz, and Simone Warzel.
\newblock Absolutely continuous spectrum for random operators on trees of
  finite cone type.
\newblock {\em Journal d'Analyse Math{\'e}matique}, 118(1):363--396, 2012.

\bibitem{kesten1959symmetric}
Harry Kesten.
\newblock Symmetric random walks on groups.
\newblock {\em Transactions of the American Mathematical Society},
  92(2):336--354, 1959.

\bibitem{klein1998extended}
Abel Klein.
\newblock Extended states in the anderson model on the bethe lattice.
\newblock {\em Advances in Mathematics}, 133(1):163--184, 1998.

\bibitem{lehner2008spectrum}
Franz Lehner, Markus Neuhauser, and Wolfgang Woess.
\newblock On the spectrum of lamplighter groups and percolation clusters.
\newblock {\em Mathematische Annalen}, 342(1):69--89, 2008.

\bibitem{lifshitz1965energy}
Il’ya~Mikhailovich Lifshitz.
\newblock Energy spectrum structure and quantum states of disordered condensed
  systems.
\newblock {\em Soviet Physics Uspekhi}, 7(4):549, 1965.

\bibitem{linnell1998analytic}
Peter~A Linnell.
\newblock Analytic versions of the zero divisor conjecture.
\newblock {\em London Mathematical Society Lecture Note Series}, pages
  209--248, 1998.

\bibitem{muantoiu2007spectral}
M~M{\u{a}}ntoiu and R~Tiedra De~Aldecoa.
\newblock Spectral analysis for convolution operators on locally compact
  groups.
\newblock {\em Journal of Functional Analysis}, 253(2):675--691, 2007.

\bibitem{mohar1989survey}
Bojan Mohar and Wolfgang Woess.
\newblock A survey on spectra of infinite graphs.
\newblock {\em Bulletin of the London Mathematical Society}, 21(3):209--234,
  1989.

\bibitem{pittet2000stability}
Ch~Pittet and Laurent Saloff-Coste.
\newblock On the stability of the behavior of random walks on groups.
\newblock {\em Journal of Geometric Analysis}, 10(4), 2000.

\bibitem{reed1972methods}
Michael Reed, Barry Simon, Barry Simon, and Barry Simon.
\newblock {\em Methods of modern mathematical physics}, volume~1.
\newblock Elsevier, 1972.

\bibitem{simon1985lifschitz}
Barry Simon.
\newblock Lifschitz tails for the anderson model.
\newblock {\em Journal of statistical physics}, 38:65--76, 1985.

\bibitem{stolz2011introduction}
G{\"u}nter Stolz.
\newblock An introduction to the mathematics of anderson localization.
\newblock {\em Entropy and the quantum II. Contemp. Math}, 552:71--108, 2011.

\bibitem{veselic2008existence}
Ivan Veseli{\'c}.
\newblock {\em Existence and regularity properties of the integrated density of
  states of random Schr{\"o}dinger operators}, volume 1917.
\newblock Springer, 2008.

\end{thebibliography}

\bigskip

\noindent Adam Arras, \\
HUN-REN Alfr\'ed R\'enyi Institute of Mathematics, \\Budapest, Hungary,\\ {\tt arras@renyi.hu}

\end{document}